\documentclass{amsart}

 \newtheorem{thm}{Theorem}[section]
 \newtheorem{cor}[thm]{Corollary}

 \theoremstyle{definition}
 
 \theoremstyle{remark}
 
 \newtheorem{ex}[thm]{Example}
 \numberwithin{equation}{section}
\newcommand{\A}{X}
\newcommand{\B}{Y}
\newcommand{\E}{P}
\newcommand{\F}{Q}
\newcommand{\R}{\mathcal{R}}
\newcommand{\M}{\mathcal{M}}
\newcommand{\T}{\mathcal{T}}
\newcommand{\Ss}{\mathcal{S}}
\newcommand{\PH}{\varphi:\mathcal{T}\rightarrow \mathcal{T}}
\newcommand{\de}{d:\mathcal{T}\rightarrow \mathcal{T}}
\newcommand{\del}{\delta:\mathcal{T}\rightarrow \mathcal{T}}
\begin{document}

\title[CHARACTERIZING JORDAN CENTRALIZERS AND JORDAN GENERALIZED DERIVATIONS ...]
 {CHARACTERIZING JORDAN CENTRALIZERS AND JORDAN GENERALIZED DERIVATIONS ON TRIANGULAR RINGS THROUGH ZERO PRODUCTS}

\author{ Hoger Ghahramani}

\thanks{{\scriptsize
\hskip -0.4 true cm \emph{MSC(2010)}: 16W25; 15A78; 47B47; 47B49.
\newline \emph{Keywords}: Jordan centralizer; Jordan generalized
derivation; triangular ring.\\}}

\address{Department of
Mathematics, University of Kurdistan, P. O. Box 416, Sanandaj,
Iran.}

\email{h.ghahramani@uok.ac.ir; hoger.ghahramani@yahoo.com}

\address{}

\email{}

\thanks{}

\thanks{}

\subjclass{}

\keywords{}

\date{}

\dedicatory{}

\commby{}


\begin{abstract}
Let $\T$ be a $2$-torsion free triangular ring and let
$\varphi:\T\rightarrow \T$ be an additive map. We prove that if
$\A \varphi(\B)+\varphi(\B)\A=0$ whenever $\A,\B\in \T$ are such
that $\A\B=\B\A=0$, then $\varphi$ is a centralizer. It is also
shown that if $\tau:\T\rightarrow \T$ is an additive map
satisfying $\label{t2} X,Y\in \T, \quad XY=YX=0\Rightarrow X
\tau(Y)+\delta(X)Y+Y\delta(X)+\tau(Y)X=0$, where
$\delta:\T\rightarrow \T $ is an additive map satisfies $X,Y\in
\T, \quad XY=YX=0\Rightarrow X
\delta(Y)+\delta(X)Y+Y\delta(X)+\delta(Y)X=0$, then
$\tau(\A)=d(\A)+\A \tau(\textbf{1})$, where $d:\T\rightarrow \T$
is a derivation and $\tau(\textbf{1})$ lies in the centre of the
$\T$. By applying this results we obtain some corollaries
concerning (Jordan) centralizers and (Jordan) derivations on
triangular rings.
\end{abstract}

\maketitle

\section{Introduction}
Throughout this paper all rings are associative. Let $ \R$ be a
ring with centre $Z(\R) $. Recall that an additive map
$\varphi:\R \rightarrow \R$ is said to be a \emph{centralizer} if
$\varphi(xy)=x\varphi(y)=\varphi(x)y$ for each $x,y\in \R$. In
case $\R$ has a unity $1$, $\varphi$ is a centralizer if and only
if $\varphi(x)=\varphi(1)x$ for each $x\in \R$, where $
\varphi(1)\in Z(\R)$. We say that $\varphi$ is a \emph{Jordan
centralizer} if $\varphi(xy+yx)=x\varphi(y)+\varphi(y)x$ for all
$x,y\in \R$. Clearly, each centralizer is a Jordan centralizer.
The converse is, in general, not true (see \cite{Gha}, Example
2.6).
\par
In general, the question under what conditions that a map becomes
a centralizer attracted much attention of mathematicians. Vukman
\cite{Vuk} has showed that an additive map
$\varphi:\mathcal{R}\rightarrow \mathcal{R}$, where $\mathcal{R}$
is a $2$-torsion free semiprime ring, with the property that
$2\varphi(x^{2})=x\varphi(x)+\varphi(x)x$ for all $x\in \R$, is a
centralizer. Hence any Jordan centralizer on a $2$-torsion free
semiprime ring is a centralizer. Benkovi$\breve{c}$ et al.
\cite{Ben} have proved that if there exists an additive mapping
$\varphi:\mathcal{R}\rightarrow \mathcal{R}$, where $\R$ is a
prime ring with suitable characteristic restrictions, satisfying
the relation $2\varphi(x^{n+1}) = \varphi(x)x^{n}+x^{n}\varphi(x)$
for all $x\in \R$ and some fixed integer $n$, then $ \varphi$ is a
centralizer. Vukman \cite{Vuk2} has showed the following result.
If $\varphi:\R \rightarrow \R $ is an additive mapping, where $\R$
is a 2-torsion free semiprime ring, satisfying the relation
$\varphi(xyx) = x\varphi(y)x$ for all pairs $x\in \R$, then
$\varphi$ is a centralizer. In \cite{Gha} the author study
continuous linear maps behaving like Jordan centralizers when
acting on unit-product elements on Banach algebras, that is, a map
$\varphi:\mathcal{A}\rightarrow \mathcal{A} $ satisfying $a,b\in
\mathcal{A}, \quad ab=ba=1 \Rightarrow
a\varphi(b)+\varphi(b)a=2\varphi(1) $, where $ \mathcal{A}$ is
unital Banach algebra. For results concerning centralizers on
rings and algebras we refer to \cite{Gha, Guo} where further
references can be found.
\par
In this paper, Motivated by \cite{Gha}, we consider the subsequent
condition on an additive map $\varphi:\T \rightarrow \T $:
\begin{equation}\label{t1}
X,Y \in \T, \quad XY=YX=0\Rightarrow X\varphi(Y)+\varphi(Y)X=0,
\end{equation}
where $\T$ is a triangular ring.
\par
Let $\R$ be a unital ring. Recall that an additive map
$\delta:\mathcal{R}\rightarrow \mathcal{R}$ is said to be a
\emph{Jordan derivation} (or \emph{generalized Jordan
derivation}) if
$\delta(xy+yx)=\delta(x)y+x\delta(y)+\delta(y)x+y\delta(x)$ (or
$\delta(xy+yx)=\delta(x)y+x\delta(y)+\delta(y)x+y\delta(x)-x\delta(1)y-y\delta(1)x$)
for all $x,y\in \mathcal{R}$. It is called a \emph{derivation} (or
\emph{generalized derivation}) if
$\delta(xy)=\delta(x)y+x\delta(y)$ (or
$\delta(xy)=\delta(x)y+x\delta(y)-x\delta(1)y$) for all $x,y\in
\mathcal{R}$. Clearly, each (generalized) derivation is a
(generalized) Jordan derivation. The converse is, in general, not
true.
\par
The question under what conditions that a map becomes a
(generalized) derivation or (generalized) Jordan derivation
attracted much attention of mathematicians. Herstein\cite{Her}
proved that every Jordan derivation from a $2$-torsion free prime
ring into itself is a derivation. Bre$\check{\textrm{s}}$ar
\cite{Bre} showed that every Jordan derivation from a $2$-torsion
free semiprime ring into itself is a derivation. By a classical
result of Jacobson and Rickart \cite{Jac} every Jordan derivation
on a full matrix ring over a $2$-torsion free unital ring is a
derivation. Benkovi$\check{\textrm{c}}$ \cite{Ben0} determined
Jordan derivations on triangular matrices over commutative rings
and proved that every Jordan derivation from the algebra of all
upper triangular matrices into its arbitrary bimodule is the sum
of a derivation and an antiderivation. Zhang and Yu \cite{Zh2}
showed that every Jordan derivation of triangular algebras is a
derivation, so every Jordan derivation from the algebra of all
upper triangular matrices into itself is a derivation. For more
studies concerning Jordan derivations we refer the reader to
\cite{Gh0, Gh} and the references therein. Recently, there have
been a number of papers on the study of conditions under which
(generalized) derivation or (generalized) Jordan derivation of
rings or algebras can be completely determined by the action on
some elements concerning products. For instance, see \cite{Gha0,
Gh} and the references therein.
\par
Motivated by \cite{Li}, we will call an additive map
$\tau:\R\rightarrow \R$ a \emph{Jordan generalized derivation via
a Jordan derivation $\delta$} if there exists a Jordan derivation
$\delta:\R\rightarrow \R$ such that
$\tau(xy+yx)=x\tau(y)+\delta(x)y+\tau(y)x+y\delta(x)$ for all
$x,y\in \mathcal{R}$. Obviously, the definition of a generalized
Jordan derivation is generally not equivalent to that of Jordan
generalized derivation. Each Jordan derivation is a Jordan
generalized derivation and any generalized derivation is a
generalized Jordan derivation, but generalized derivations are
not necessarily Jordan generalized derivations (see
Example~\ref{ee}).
\par
In this article, we also consider the following conditions on an
additive map $\tau: \T\rightarrow \T$:
\begin{equation}\label{t2}
X,Y\in \T, \quad XY=YX=0\Rightarrow X
\tau(Y)+\delta(X)Y+Y\delta(X)+\tau(Y)X=0,
\end{equation}
where $\T$ is a triangular ring and $\delta:\T\rightarrow \T $ is
an additive map satisfying
\[X,Y\in \T, \quad XY=YX=0\Rightarrow X
\delta(Y)+\delta(X)Y+Y\delta(X)+\delta(Y)X=0.\] This article is
organized as follows. Suppose that $\T$ is a 2-torsion free
triangular ring with unity matrix $ \textbf{1}$. In Section 2, we
give some preliminaries. Section 3 is devoted to characterizing
the Jordan centralizers by acting on zero products on triangular
rings. Indeed, we show that each additive map $ \varphi$ on $\T$
satisfying \eqref{t1} is a centralizer. Then by applying this
result we see that each Jordan centralizer on $\T$ is a
centralizer. Also we obtain that if $\varphi:\T\rightarrow \T$ is
an additive map satisfying $\varphi(XYX)=X\varphi(Y)X$ for all
$X,Y\in \T$, then $\varphi$ is a centralizer. In Section 4 we
prove that condition \eqref{t2} imply $\tau$ is of the form
$\tau(\A)=d(\A)+\A \tau(\textbf{1})$ for each $\A\in \T$, where
$d: \T\rightarrow \T$ is a derivation, $\tau(\textbf{1})\in
Z(\T)$. As applications of the this result, we show that every
Jordan derivation of the trivial extension of $\T$ by $\T$ is a
derivation.
\section{Preliminaries}
Recall that a \emph{triangular ring} $Tri(\R,\M,\Ss)$ is a ring of
the form
\[ Tri(\R,\M,\Ss):=\bigg \{ \begin{pmatrix}
  r & m \\
  0 & s
\end{pmatrix}\,\bigg | \, r \in \R,\, s\in \Ss, \, m\in \M\bigg \} \]
under the usual matrix operations, where $\R$ and $\Ss$ are unital
rings and $\M$ is a unital $(\R, \Ss)$-bimodule which is faithful
as a left $\R$-module as well as a right $\Ss$-module. The most
important examples of triangular rings are upper triangular
matrices over a ring $\mathcal{R}$, block upper triangular matrix
algebras, nest algebras over a real or a complex Banach space
$\mathcal{X}$ or a Hilbert space $\mathcal{H}$, respectively and
generalized triangular matrix algebras. Recently, there has been a
growing interest in the study of linear maps that preserve zero
products, Jordan products, commutativity, etc. and derivable
(resp., Jordan derivable, Lie derivable) maps at zero point,
etc., on triangular rings (algebras). For instance, see
\cite{Gha1} and the references therein.
\par
Throughout this paper $\R$ and $\Ss$ are unital 2-torsion free
rings, and $\M$ is a unital 2-torsion free $(\R,\Ss)$-bimodule,
which is faithful as a left $\R$-module and also as a right
$\Ss$-module. Also $\T$ denotes the triangular ring
$Tri(\R,\M,\Ss)$ which is a 2-torsion free ring. Let $1_{\R}$ and
$1_{Ss}$ be identities of the rings $\R$ and $\Ss$, respectively.
We denote the identity of the triangular ring $\T$, i.e. the
identity matrix $\begin{pmatrix}
  1_{\R} & 0 \\
  0 & 1_{\Ss}
\end{pmatrix}$ by $\textbf{1}$. Also, throughout this paper we shall use following notation
\[ P=\begin{pmatrix}
  1_{\R} & 0 \\
  0 & 0
\end{pmatrix} \quad and \quad Q=\begin{pmatrix}
 0 & 0 \\
 0 & 1_{\Ss}
\end{pmatrix}. \]
We immediately notice that $P$ and $Q$ are the standard
idempotents (i.e. $P^{2}=P$ and $Q^{2}=Q$) in $\T$ such that
$P+Q=\textbf{1}$ and $PQ=QP=0$.

\section{Characterizing Jordan centralizers through zero products}
In this section, we consider the question of characterizing
Jordan centralizers by action at zero products on triangular
rings. The results in this section are also basic to discuss the
additive maps Jordan generalized derivable through zero products.
\begin{thm}\label{cc}
Let $\PH$ be an additive map satisfying
\[ X,Y\in \T, \quad XY=YX=0\Rightarrow X \varphi(Y)+\varphi(Y)X=0.\]
Then $\varphi$ is a centralizer.
\end{thm}
\begin{proof}
Let $X$ and $Y$ be arbitrary elements in $\T$. Since
$P(QXQ)=(QXQ)P=0$, we have
\begin{equation}\label{c1}
P\varphi(QXQ)+\varphi(QXQ)P=0.
\end{equation}
Multiplying this identity by $P$ both on the left and on the
right we find $2P\varphi(QXQ)P=0$, so
\begin{equation}\label{c2}
P\varphi(QXQ)P=0.
\end{equation}
Now, multiplying \eqref{c1} from the left by $P$, from the right
by $Q$, we get
\begin{equation}\label{c3}
P\varphi(QXQ)Q=0.
\end{equation}
From $Q(PXP)=(PXP)Q=0$, we have
\[ Q\varphi(PXP)+\varphi(PXP)Q=0.\]
By this identity and using similar methods as above we obtain
\begin{equation}\label{c4}
Q\varphi(PXP)Q=0 \quad and \quad P\varphi(PXP)Q=0.
\end{equation}
Since $(P-PXQ)(Q+PXQ)=(Q+PXQ)(P-PXQ)=0$, it follows that
\begin{equation*}\label{c5}
 (P-PXQ)\varphi(Q+PXQ)+\varphi(Q+PXQ)(P-PXQ)=0.
\end{equation*}
Multiplying this identity by $P$ both on the left and on the right
and by the fact that $P\varphi(Q)P=0$, we see that
\begin{equation}\label{c6}
P\varphi(PXQ)P=0.
\end{equation}
from $(PXP-PXPYQ)(Q+PYQ)=(Q+PYQ)(PXP-PXPYQ)=0$, we have
\begin{equation}\label{c7}
 (Q+PYQ)\varphi(PXP-PXPYQ)+\varphi(PXP-PXPYQ)(Q+PYQ)=0.
\end{equation}
Letting $X=P$ and multiplying above identity by $Q$ both on the
left and on the right and by the fact that $Q\varphi(P)Q=0$, we
obtain
\begin{equation}\label{c8}
Q\varphi(PYQ)Q=0.
\end{equation}
Multiplying \eqref{c7} by $P$ on the left and by $Q$ on the right,
from  \eqref{c4}, \eqref{c6} and \eqref{c8} we arrive at
\begin{equation}\label{c9}
P\varphi(PXPYQ)Q=P\varphi(PXP)PYQ.
\end{equation}
Replacing $X$ by $P$ in above equation, we get
\begin{equation}\label{c10}
P\varphi(PYQ)Q=P\varphi(P)PYQ.
\end{equation}
So from \eqref{c9} and \eqref{c10}, it follows that
\[ P\varphi(PXP)PYQ= P\varphi(PXPYQ)Q=P\varphi(P)PXPYQ, \]
and hence $(P\varphi(PXP)P-P\varphi(P)PXP)PYQ=0$. Since $Y\in \T$
is arbitrary and $\M$ is a faithful left $\R$-module, we find
\begin{equation}\label{c11}
P\varphi(PXP)P=P\varphi(P)PXP.
\end{equation}
From $(P-PXQ)(PXQYQ+QYQ)=(PXQYQ+QYQ)(P-PXQ)=0$, we have
\[ (P-PXQ)\varphi(PXQYQ+QYQ)+\varphi(PXQYQ+QYQ)(P-PXQ)=0. \]
Multiplying this identity by $P$ on the left and by $Q$ on the
right, from \eqref{c2}, \eqref{c3}, \eqref{c6} and \eqref{c8} we
see that
\begin{equation}\label{c12}
P\varphi(PXQYQ)Q=PXQ\varphi(QYQ)Q.
\end{equation}
Replacing $Y$ by $Q$ in above equation, we get
\begin{equation}\label{c13}
P\varphi(PXQ)Q=PXQ\varphi(Q)Q.
\end{equation}
By \eqref{c12} and \eqref{c13}, using similar methods as above
and the fact that $\M$ is a faithful right $\Ss$-module, we obtain
\begin{equation}\label{c14}
Q\varphi(QYQ)Q=QYQ\varphi(Q)Q.
\end{equation}
By \eqref{c10} and \eqref{c13}, we have
\begin{equation}\label{c15}
P\varphi(P)PXQ=PXQ\varphi(Q)Q.
\end{equation}
So
\[ P\varphi(P)PXPYQ=PXPYQ\varphi(Q)Q=PXP\varphi(P)PYQ, \]
and hence
\begin{equation}\label{c16}
P\varphi(P)PXP=PXP\varphi(P)P,
\end{equation}
since $\M$ is a faithful left $\R$-module. Similarly from
\eqref{c15}, we get
\begin{equation}\label{c17}
Q\varphi(Q)QXQ=QXQ\varphi(Q)Q.
\end{equation}
Now by \eqref{c2}, \eqref{c3}, \eqref{c4}, \eqref{c15},
\eqref{c16} and \eqref{c17}, we have
\begin{equation}\label{c18}
\begin{split}
X\varphi(\textbf{1})&=PXP\varphi(P)P+PXQ\varphi(Q)Q+QXQ\varphi(Q)Q
\\&=P\varphi(P)PXP+P\varphi(P)PXQ+Q\varphi(Q)QXQ=\varphi(\textbf{1})X,
\end{split}
\end{equation}
and from \eqref{c2}, \eqref{c3}, \eqref{c4}, \eqref{c6},
\eqref{c8}, \eqref{c10}, \eqref{c11}, \eqref{c14} and \eqref{c18},
we arrive at
\begin{equation*}
\begin{split}
\varphi(X)&=P\varphi(PXP)P+P\varphi(PXQ)Q+Q\varphi(QXQ)Q
\\&
=P\varphi(P)PXP+P\varphi(P)PXQ+QXQ\varphi(Q)Q=\varphi(\textbf{1})X.
\end{split}
\end{equation*}
These results show that $\varphi$ is a centralizer.
\end{proof}
Since every Jordan centralizer satisfies the requirements in
Theorem~\ref{c1}, the following corollary is clear.
\begin{cor}\label{co1}
Suppose that $\PH$ is a Jordan centralizer. Then $\varphi$ is a
centralizer.
\end{cor}
Also from this result we can obtain the following corollary.
\begin{cor}\label{co2}
Let $\PH$ be an additive mapping satisfying the relation
\begin{equation}\label{c18}
\varphi(XYX)=X\varphi(Y)X,
\end{equation}
for all $X,Y\in \T$. Then $\varphi$ is a centralizer.
\end{cor}
\begin{proof}
Let $X$ and $Y$ be arbitrary elements in $\T$. Replacing $X$ by
$X+\textbf{1}$ in \eqref{c18} we obtain
\[
\varphi((X+\textbf{1})Y(X+\textbf{1}))=(X+\textbf{1})\varphi(Y)(X+\textbf{1}).\]
Hence from hypothesis we find
\[ \varphi(XY+YX)=X\varphi(Y)+\varphi(Y)X.\]
So $\varphi$ is a Jordan centralizer and by Corollary~\ref{co1},
it is a centralizer.
\end{proof}
\section{Characterizing Jordan generalized derivations through zero products}
In this section, we discuss the question of characterizing Jordan
generalized derivations through zero products on triangular
rings. The following is our main result.
 \begin{thm}\label{asli}
 Suppose that $\tau:\T\rightarrow \T$ is an additive map. Then
the following conditions are equivalent:
\begin{enumerate}
\item[(i)] There exist a derivation $\de$ such that $\tau(X) = d(X) +
X\tau(\textbf{1})$ for each $X\in \T$ and $\tau(\textbf{1})\in
Z(\T)$.
\item[(ii)]
\begin{equation*}
 X,Y\in \T, \quad XY=YX=0\Rightarrow X
\tau(Y)+\tau(X)Y+Y\tau(X)+\tau(Y)X=0.
\end{equation*}
\item[(iii)] $ \tau$ is a Jordan generalized derivation via a
Jordan derivation $\delta$.
\item[(iv)] There is an additive map $\del$ such that
\begin{equation*}
 X,Y\in \T, \quad XY=YX=0\Rightarrow X
\delta(Y)+\delta(X)Y+Y\delta(X)+\delta(Y)X=0.
\end{equation*}
and
\begin{equation*}
 X,Y\in \T, \quad XY=YX=0\Rightarrow X
\tau(Y)+\delta(X)Y+Y\delta(X)+\tau(Y)X=0.
\end{equation*}
\end{enumerate}
\end{thm}
\begin{proof}
The proof of $(i)\Rightarrow (ii) \Rightarrow (iv)$ and
$(i)\Rightarrow (iii) \Rightarrow (iv)$ is routine.
\par
$(iv)\Rightarrow (i)$: We first show that there exist a
derivation $\de$ such that $\delta(X) = d(X) +
X\delta(\textbf{1})$ for each $X\in \T$ and $\delta(\textbf{1})\in
Z(\T)$.
\par
Let $X$, $Y$ and $Z$ be arbitrary elements in $\T$. Let
$W=P\delta(P)Q$. Define $\Delta:\T\rightarrow \T$ by
$\Delta(X)=\delta(X)-WX+XW$. Then $\Delta$ is an additive mapping
which satisfies
\begin{equation*}
 X,Y\in \T, \quad XY=YX=0\Rightarrow X
\Delta(Y)+\Delta(X)Y+Y\Delta(X)+\Delta(Y)X=0.
\end{equation*}
Moreover $P\Delta(P)Q=0$.
\par
Since $\E(\F \A \F )=(\F \A \F)\E=0$, we have
\begin{equation}\label{1}
\Delta(\E)\F \A \F +\E \Delta(\F \A \F)+\Delta(\F \A \F )\E +\F
\A \F \Delta(\E)=0.
\end{equation}
Multiplying this identity by $\E$ both on the left and on the
right we find $2\E \Delta(\F \A \F)\E=0$ so $\E
\Delta(\F\A\F)\E=0$. Now, multiplying the \eqref{1} from the left
by $\E$, from the right by $\F$ and by the fact that
$\E\Delta(\E)\F=0$, we get $\E\Delta(\F\A\F)\F=0$. Therefore, from
above equations we arrive at
\begin{equation}\label{j2}
\Delta(\F\A\F )=\F\Delta(\F\A\F)\F
\end{equation}
We have $(\E\A\E)\F=\F(\E\A\E)=0$. Thus
\begin{equation}\label{3}
\Delta(\E\A\E)\F+\E\A\E\Delta(\F)+\Delta(\F)(\E\A\E)+\F\Delta(\E\A\E)=0.
\end{equation}
By \eqref{j2}, \eqref{3} and using similar methods as above we
obtain
\begin{equation}\label{j3}
\Delta(\E\A\E)=\E\Delta(\E\A\E)\E.
\end{equation}
We have
$(\E\A\E+\E\A\E\B\F)(\F-\E\B\F)=(\F-\E\B\F)(\E\A\E+\E\A\E\B\F)=0$
and so
\begin{equation}\label{53}
\begin{split}
&\Delta(\E\A\E+\E\A\E\B\F)(\F-\E\B\F)+(\E\A\E+\E\A\E\B\F)\Delta(\F-\E\B\F)\\&+\Delta(\F-\E\B\F)(\E\A\E+\E\A\E\B\F)+(\F-\E\B\F)\Delta(\E\A\E+\E\A\E\B\F)=0.
\end{split}
\end{equation}
Multiplying \eqref{53} by $\E$ both on the left and on the right
and replacing $\A$ by $\E$, from \eqref{j2} we get
$\E\Delta(\E\B\F)\E=0$. Now multiplying \eqref{53} by $\F$ both
on the left and on the right, by \eqref{j3} and a similar
arguments as above we find $\F\Delta(\E\B\F)\F=0$. From previous
equations it follows that
\begin{equation}\label{j4}
\Delta(\E\B\F)=\E\Delta(\E\B\F)\F.
\end{equation}
Multiplying \eqref{53} by $\E$ on the left and by $\F$ on the
right, from \eqref{j2}, \eqref{j3} and \eqref{j4} we obtain
\begin{equation}\label{j5}
\E\Delta(\E\A\E\B\F)\F=\E\A\E\Delta(\E\B\F)\F+\E\Delta(\E\A\E)\E\B\F-\E\A\E\B\F\Delta(\F)\F.
\end{equation}
Replacing $\A$ by $\E$ in above identity, we get
\begin{equation}\label{71}
\E\Delta(\E)\E\B\F=\E\B\F\Delta(\F)\F
\end{equation}
Since
$(\E+\E\A\F)(\F\B\F-\E\A\F\B\F)=(\F\B\F-\E\A\F\B\F)(\E+\E\A\F)=0$,
we have
\begin{equation*}
\begin{split}
&\Delta(\E+\E\A\F)(\F\B\F-\E\A\F\B\F)+(\E+\E\A\F)\Delta(\F\B\F-\E\A\F\B\F)\\&+\Delta(\F\B\F-\E\A\F\B\F)(\E+\E\A\F)+(\F\B\F-\E\A\F\B\F)\Delta(\E+\E\A\F)=0
\end{split}
\end{equation*}
Multiplying this identity by $\E$ on the left and by $\F$ on the
right, from \eqref{j2}, \eqref{j3}, \eqref{j4} and \eqref{71} we
arrive at
\begin{equation}\label{j6}
\E\Delta(\E\A\F\B\F)\F=\E\Delta(\E\A\F)\F\B\F+\E\A\F\Delta(\F\B\F)\F-\E\A\F\Delta(\F)\F\B\F.
\end{equation}
From \eqref{j5} we see that
\begin{equation*}
\begin{split}
\E\Delta(\E\A\E\B\E Z\F)\F&=\E\A\E\B\E\Delta(\E Z\F)\F\\ &
+\E\Delta(\E\A\E\B\E)\E Z\F-\E\A\E\B\E Z\F\Delta(\F)\F.
\end{split}
\end{equation*}
On the other hand,
\begin{equation*}
\begin{split}
\E\Delta(\E\A\E\B\E Z\F)\F&= \E\A\E\Delta(\E\B\E Z\F)\F \\&
+\E\Delta(\E\A\E)\E\B\E Z\F -\E\A\E\B\E Z\F\Delta(\F)\F\\&
=\E\A\E\B\E\Delta(\E Z\F)\F\\&+\E\A\E\Delta(\E\B\E)\E Z\F
-\E\A\E\B\E Z\F \Delta(\F)\F\\&+\E\Delta(\E\A\E)\E\B\E Z\F
-\E\A\E\B\E Z\F\Delta(\F)\F.
\end{split}
\end{equation*}
By comparing the two expressions for $\E\Delta(\E\A\E\B\E Z\F )\E
Z\F $, \eqref{71} and the fact that $\M$ is a faithful left
$\R$-module, yields
\begin{equation}\label{j7}
\E\Delta(\E\A\E\B\E)\E=\E\A\E\Delta(\E\B\E)\E+\E\Delta(\E\A\E)\E\B\E-\E\A\E\Delta(\E)\E\B\E.
\end{equation}
From the fact that $\M$ is a faithful right $\Ss$-module,
\eqref{j6} and a proof similar to above, we find
\begin{equation}\label{j8}
\F\Delta(\F\A\F\B\F)\F=\F\Delta(\F\A\F)\F\B\F+\F\A\F\Delta(\F\B\F)\F-\F\A\F\Delta(\F)\F\B\F.
\end{equation}
By \eqref{71} we have
\begin{equation*}
\begin{split}
&\E\A\E\Delta(\E)\E\B\F
=\E\A\E\B\F\Delta(\F)\F=\E\Delta(\E)\E\A\E\B\F
\\& \textrm{and}
\\&
\E\B\F\Delta(\F)\F\A\F=\E\Delta(\E)\E\B\F\A\F=\E\B\F\A\F\Delta(\F)\F.
\end{split}
\end{equation*}
So by the fact that $\M$ is a faithful left $\R$-module and a
faithful right $\Ss$-module, we have
\begin{equation}\label{101}
\E\A\E\Delta(\E)\E=\E\Delta(\E)\E\A\E, \quad
\F\Delta(\F)\F\A\F=\F\A\F\Delta(\F)\F
\end{equation}
By \eqref{j2} and \eqref{j3}  we have
$\Delta(\textbf{1})=\E\Delta(\E)\E+\F\Delta(\F)\F$. From this
identity and \eqref{71}, \eqref{101} we arrive at
\begin{equation}\label{j9}
\begin{split}
\A\Delta(\textbf{1})&=\E\A\E\Delta(\textbf{1})+\E\A\F\Delta(\textbf{1})+\F\A\F\Delta(\textbf{1})
\\&=\E\A\E\Delta(\E)\E+\E\A\F\Delta(\F)\F+\F\A\F\Delta(\F)\F
\\&=\E\Delta(\E)\E\A\E+\E\Delta(\E)\E\A\F+\F\Delta(\F)\F\A\F
\\&=\Delta(\textbf{1})\E\A\E+\Delta(\textbf{1})\E\A\F+\Delta(\textbf{1})\F\A\F
\\&=\Delta(\textbf{1})\A.
\end{split}
\end{equation}
\par We have $\delta(\textbf{1})=\Delta(\textbf{1})$ and hence from \eqref{j9} we find that $\delta(\textbf{1})\in Z(\mathcal{M})$. Since
$\A\B=\E\A\E\B\E+\E\A\E\B\F+\E\A\F\B\F+\F\A\F\B\F$ for any $\A,\B
\in \T$, by \eqref{j2}, \eqref{j3}, \eqref{j4},\eqref{j5},
\eqref{j6}, \eqref{j7}, \eqref{j8} and \eqref{j9}, it follows that
the mapping $\Delta^{\prime}:\T\rightarrow \T$ given by
$\Delta^{\prime}(\A)=\Delta(\A)-\A\Delta(\textbf{1})$ is a
derivation. So the mapping $\de$ given by
$d(\A)=\Delta^{\prime}(\A)+(WX-XW)$ is a derivation and we have
$\delta(\A)=d(\A)+\A \delta(\textbf{1})$ for all $\A \in \T$.
\par
Now define $\PH$ by $\varphi=\tau-\delta $. By hypothesis
$\varphi$ is an additive map satisfying
\[ X,Y\in \T, \quad XY=YX=0\Rightarrow X \varphi(Y)+\varphi(Y)X=0.\]
From Theorem~\ref{c1}, $\varphi$ is a centralizer. From above
results we have
\[ \tau(X)=\delta(X)+\varphi(X)=d(X)+X\delta(\textbf{1})+X\varphi(\textbf{1})=d(X)+X\tau(\textbf{1}) \,\,\, X\in \T,\]
where $\de$ is a derivation. Since $\varphi(\textbf{1}), \delta(
\textbf{1})\in Z(\T) $, it follows that
$\tau(\textbf{1})=\varphi(\textbf{1})+ \delta( \textbf{1})\in
Z(\T)$. The proof is now complete.
\end{proof}
Now, we can give an example which shows that generalized
derivations are not necessarily Jordan generalized derivations.
\begin{ex}\label{ee}
Suppose that $0 \neq m\in \M$ is an arbitrary element and
$X=\begin{pmatrix}
  0 & m \\
  0 & 0
\end{pmatrix}\in \T$. Define an additive map $\delta:\T \rightarrow
\T$ by $\delta(T)=TX$. Since $\delta( \textbf{1})=X$ is not in
$Z(\T)$, by Theorem~\ref{asli}, $\delta$ is not a Jordan
generalized derivation, while by a straightforward calculation one
can prove that $\delta$ is a generalized derivation.
\end{ex}
If $\del$ is a Jordan derivation, then $\delta$ satisfies the
requirements in Theorem~\ref{asli}(iii) and
$\delta(\textbf{1})=0$. So $\delta$ is a derivation and hence
Theorem~\ref{asli} generalizes the main result of \cite{Zh2}.
\par
Given a ring $\mathcal{A}$ and an $\mathcal{A}$-bimodule
$\mathcal{M}$, the \emph{trivial extension} of $\mathcal{A}$ by
$\mathcal{M}$ is the ring
$T(\mathcal{A},\mathcal{M})=\mathcal{A}\oplus\mathcal{M}$ with
the usual addition and the following multiplication:
\[ (a_{1},m_{1})(a_{2},m_{2})=(a_{1}a_{2}, a_{1}m_{2}+m_{1}a_{2}).
\]
Let $\mathcal{A}$ be a $2$-torsion free unital ring. Suppose that
each additive mapping $\delta:\mathcal{A}\rightarrow \mathcal{A}$
satisfying
\[  a,b\in \mathcal{A}, \quad ab=ba=0\Rightarrow a
\delta(b)+\delta(a)b+b\delta(a)+\delta(b)a=0,\] is a generalized
derivation with $\delta(1)\in Z(\mathcal{A})$. Let
$T(\mathcal{A},\mathcal{A})$ be the trivial extension of
$\mathcal{A}$ by $\mathcal{A}$. Then by \cite[Lemma 4.3.]{Gh}
every Jordan derivation from $T(\mathcal{A},\mathcal{A})$ into
itself is a derivation. So from this observation and
Theorem~\ref{asli} we have the next corollary.
\begin{cor}\label{tra}
Every Jordan derivation from $T(\T,\T)$ into itself is a
derivation.
\end{cor}
\bibliographystyle{amsplain}
\bibliography{xbib}

\end{document}